\documentclass[letterpaper]{article}

\usepackage[left=4cm,right=4cm,bottom=3.3cm,top=2.5cm]{geometry}
\usepackage{amsmath,amsthm,amsfonts,amssymb,amsxtra,bm} 
\usepackage{hyperref}
\usepackage{cite}
\usepackage{graphicx}
\usepackage{enumitem}
\usepackage{cite}
\usepackage{psfrag}
\usepackage{caption}

\newtheorem{theorem}{Theorem}

\newtheorem{lemma}{Lemma}

\author{Ondrej Slu\v{c}iak\thanks{This work was funded by the Austrian Science Fund (FWF) in Project
NFN SISE S10611.}}
\date{\vspace*{-0.3cm}\small(\texttt{osluciak@nt.tuwien.ac.at})\vspace*{-0.4cm}}

\title{On the roots of $a^x+a^{-x}=x$}

\hypersetup{
	unicode,
	bookmarksnumbered,
	bookmarksopen=false,
	bookmarksopenlevel=1,
	pdfauthor={O.Slu\v{c}iak},
	pdftitle={On the roots of ax+a{-x}=x},
	pdfkeywords={},
	pdfsubject={}
}

\newcommand{\arcsinh}{\mathrm{arcsinh}}

\begin{document}
\maketitle

\begin{abstract}
We provide an analytical closed-form solution of the exponential equation\break $a^x+a^{-x}=x$ for a specific value $a$, discuss the number of roots in general case, and provide bounds on these roots.
\end{abstract}

\section{Introduction}

It is well-known that the solution of equations of type $a^x=x$ is given by the so-called Lambert $W$ function\cite{Corless96}, defined as the inverse function of $f(x)=xe^x$, i.e.,\break $W(x)=f^{-1}(x)$. 
Thus, in case of $a^x=x$ the solution is then $x=-\frac{W(-\ln a)}{\ln(a)}$. Although this function has been known since 18th century, its usefulness and applications started to be apparent only in few last decades\cite{Corless96,Valluri00,Jenn02,Scott06,Chatzigeorgiou13,Yi11,Ban09,Hwang05,ChapeauBlondeau02}. 

Nevertheless, its applicability is limited only to few cases and therefore, already in slightly more complicated scenarios, we have to take a different approach. 
In what follows, we analyze the solutions of equation
\begin{equation}
a^x+a^{-x}=x\label{eq:main}
\end{equation}
where $a\in\mathbb{R}^+_0, x\in\mathbb{R}$.

\section{Main result}

First of all, let us start with very simple options.
\begin{enumerate}
\item In case $a=1$, there is exactly one solution to Eq.~\eqref{eq:main}, $x=2$.
\item In case $a=0$, there is exactly one solution to Eq.~\eqref{eq:main}, $x=0$.
\end{enumerate}

\noindent Let us move directly to the main result.
\begin{theorem}\label{theorem}
In case $a\in\mathbb{R}^+\backslash\{1\}$, there are three possible cases to Eq.~\eqref{eq:main}
\begin{enumerate}
\item If $a\in (0,e^{-\frac{1}{2\sinh q}}) \cup (e^{\frac{1}{2\sinh q}},\infty)$ there exists no solution $x$.
\item If $a = e^{-\frac{1}{2\sinh q}}$ or $a = e^{\frac{1}{2\sinh q}}$, there exists exactly one solution, $x^\dagger = 2\cosh q$.
\item If $a\in (e^{-\frac{1}{2\sinh q}}, e^{\frac{1}{2\sinh q}})$, there exist exactly two solutions.
\end{enumerate}
Constant $q$ is a solution to equation
\begin{equation*}
\coth q = q
\end{equation*}
and is approximately $q\approx 1.19967864\dots$.
\end{theorem}

Constant $q$ can be found in literature in relation to so-called Laplace limit\cite{Finch03}.
Numerically, the minimum and maximum $a$ for which there exists at least one solution is $a=e^{-\frac{1}{2\sinh q}}\approx 0.71793825$ and $a=e^{\frac{1}{2\sinh q}}\approx 1.39287744$, with $x^\dagger=2\cosh q\approx 3.62034$.

\begin{figure}[t]\centering\vspace{-0.1cm}
\psfrag{a}[Bc][Bc][1][0]{$-10$}
\psfrag{b}[Bc][Bc][1][0]{$-8$}
\psfrag{c}[Bc][Bc][1][0]{$-6$}
\psfrag{d}[Bc][Bc][1][0]{$-4$}
\psfrag{e}[Bc][Bc][1][0]{$-2$}
\psfrag{f}[Bc][Bc][1][0]{$0$}
\psfrag{g}[Bc][Bc][1][0]{$2$}
\psfrag{h}[Bc][Bc][1][0]{$4$}
\psfrag{i}[Bc][Bc][1][0]{$6$}
\psfrag{j}[Bc][Bc][1][0]{$8$}
\psfrag{k}[Bc][Bc][1][0]{$10$}
\psfrag{A}[Br][Br][1][0]{$-4$}
\psfrag{B}[Br][Br][1][0]{$-2$}
\psfrag{C}[Br][Br][1][0]{$0$}
\psfrag{D}[Br][Br][1][0]{$2$}
\psfrag{E}[Br][Br][1][0]{$4$}
\psfrag{F}[Br][Br][1][0]{$6$}
\psfrag{G}[Br][Br][1][0]{$8$}
\psfrag{H}[Br][Br][1][0]{$10$}
\psfrag{R}[Bl][Bl][1][0]{$a>a^\star$}
\psfrag{S}[Bl][Bl][1][0]{$a<a^\star$}
\psfrag{y}{}
\psfrag{x}[Bc][Bc][1][0]{$x$}
\includegraphics[width=9.5cm]{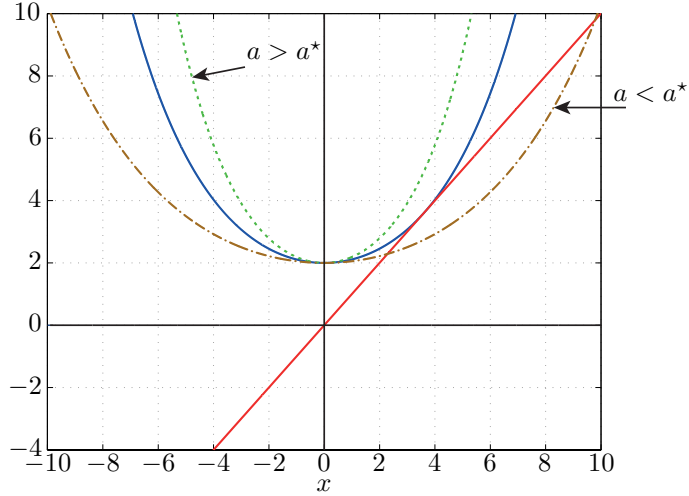}
\captionsetup{skip=0.1cm}
\caption{Intersection (roots) of $f_1(x)=2\coth(x\ln a)$ and $f_2(x)=x$ (red line). If $a=a^\star=e^{\frac{1}{2\sinh q}}$ (blue curve) there is one root $x=2\cosh q\approx 3.62034$.}\label{fig}
\vspace{-0.2cm}
\end{figure}
\begin{proof}
Let us define function $f(x)$
\begin{equation}\label{eq:mainf}
f(x) = a^x+a^{-x}-x,
\end{equation}
which can be equivalently written as
\begin{equation}\label{eq:main2}
f(x) = 2\cosh (x\ln a) -x.
\end{equation}

By minimizing, we obtain
\begin{equation*}
f^\prime(x) = 2\ln a\sinh(x\ln a)-1.
\end{equation*}
Setting $f^\prime(x)=0$, we find $x^\star = \frac{1}{\ln a}\arcsinh\frac{1}{2\ln a}$. Since $f^{\prime\prime}(x)>0, \forall x$ ($f(x)$ is convex), we see that $x^\star$ is a minimum.

Plugging $x^\star$ into Eq.~\eqref{eq:main2} we observe that
\begin{equation*}
f(x) = 2\cosh\left(\arcsinh\frac{1}{2\ln a}\right)-\frac{1}{\ln a}\arcsinh\frac{1}{2\ln a},
\end{equation*}
and setting $f(x)=0$ and substituing $q=\arcsinh\frac{1}{2 \ln a}$ we obtain
\[
\frac{\cosh q}{\sinh q} = q.
\]
From literature\cite{Finch03} we know that $q\approx  1.19967864$ and $q\approx  -1.19967864$.
Thus, the minimum is obtained when $a = e^{-\frac{1}{2\sinh q}}$ or $a = e^{\frac{1}{2\sinh q}}$ and takes the value of $x=2\cosh q$.

Furthermore, if $1<a<e^{\frac{1}{2\sinh q}}$, then $x^\star$ gets larger\footnote{Since $\cosh(-x)=\cosh(x)$, the same holds also for $e^{-\frac{1}{2\sinh q}}<a<1$.}, thus from the convexity of function of $f(x)$ and mean value theorem we know that $f_1(x)=2\cosh(x \ln a)$ intersects $f_2(x)=x$ in two points (see Fig.~\ref{fig}).

Converse holds true for all other $a\in(0,e^{-\frac{1}{2\sinh q}}) \cup (e^{\frac{1}{2\sinh q}},\infty)$.
\end{proof}

\vspace{0.4cm}
\textbf{Remark:}
Note that if we substitute in Eq.~\eqref{eq:main2} $y=x\ln a$, then $f(y)=2\cosh(y)-\frac{1}{\ln a}y$. Thus, in similar fashion as in Fig.~\ref{fig} we would fix function $2\cosh(y)$ and investigate the slope of $\frac{1}{\ln a}y$. In such a way it is more visible that in the non-trivial case if $a\to 0$, line $\frac{1}{\ln a}y$ becomes vertical, and thus there is again only one intersection point with $\cosh(y)$ at $x=0$.

\vspace{0.3cm}
As a consequence of Theorem~\ref{theorem}, we have the following lemmas.

\begin{figure}[t!]\centering
\psfrag{a}[Bc][Bc][1][0]{$0$}
\psfrag{b}[Bc][Bc][1][0]{$2$}
\psfrag{c}[Bc][Bc][1][0]{$4$}
\psfrag{d}[Bc][Bc][1][0]{$6$}
\psfrag{e}[Bc][Bc][1][0]{$8$}
\psfrag{f}[Bc][Bc][1][0]{$10$}
\psfrag{g}[Bc][Bc][1][0]{$12$}
\psfrag{h}[Bc][Bc][1][0]{$14$}
\psfrag{A}[Br][Br][1][0]{$-4$}
\psfrag{B}[Br][Br][1][0]{$-2$}
\psfrag{C}[Br][Br][1][0]{$0$}
\psfrag{D}[Br][Br][1][0]{$2$}
\psfrag{E}[Br][Br][1][0]{$4$}
\psfrag{F}[Br][Br][1][0]{$6$}
\psfrag{G}[Br][Br][1][0]{$8$}
\psfrag{H}[Br][Br][1][0]{$10$}
\psfrag{y}[Bc][Bc][1][0]{$f(x)$}
\psfrag{x}[Bc][Bc][1][0]{$x$}
\includegraphics[width=9.5cm]{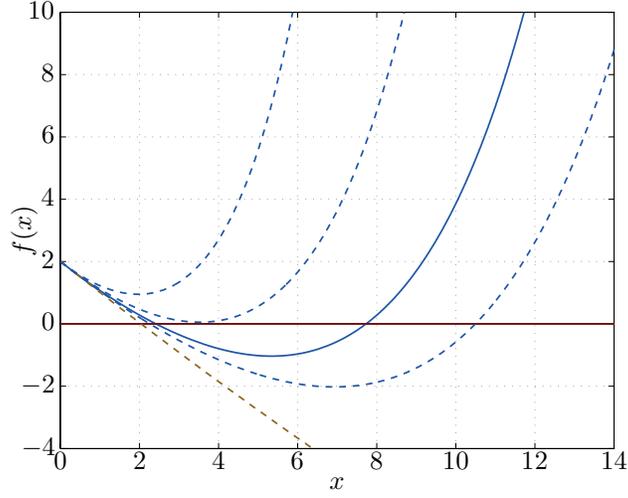}
\captionsetup{skip=0.1cm}
\caption{Function $f(x)$ (Eq.~\eqref{eq:main2}) for various $a$. Observe that root $x_1>2$ for all $a$, and that minimum $x^\star$ is closer to $x_2$ than to $x_1$.}\label{fig2}
\end{figure}
\begin{lemma}
It holds that $f(x)\geq 2-x$ ($\forall x$)(cf. Eq.~\eqref{eq:main2}).\label{lem1}
\end{lemma}

\begin{proof}
It is straightforward that if $a=1$, then $f(x)=a^x+a^{-x}-x=2-x$ is the minimal solution. Thus, $f(x)\geq 2-x$ (see Fig.~\ref{fig2}). 
\end{proof}

\vspace{0cm}
\begin{lemma}\label{lem2}
If $a\in (e^{-\frac{1}{2\sinh q}}, e^{\frac{1}{2\sinh q}})$ there exist exactly two roots $x_1$ and $x_2$ of Eq.~\eqref{eq:mainf} which are both positive, and lie in the intervals $x_1\in(2,2\cosh q)\approx(2,3.62034)$, and, $x_2\in(\frac{1}{\ln a}\arcsinh\frac{1}{2\ln a},\frac{2}{\ln a}\arcsinh\frac{1}{2\ln a}-2)$.\\[-0.2cm]

More precisely, if we know $x_1$, then $x_2\in(\frac{3}{2}\frac{1}{\ln a}\arcsinh\frac{1}{2\ln a}-\frac{1}{2}x_1,\frac{2}{\ln a}\arcsinh\frac{1}{2\ln a}-x_1)$.
\end{lemma}

\begin{proof}Let $x_1<x_2$.
The lower bound on $x_1$ follows from Lemma~\ref{lem1}. The upper bound from the fact that if $x_1=2\cosh q$ then $x_1$ is the simple root.

The lower bound on $x_2$ follows from convexity of the function and the fact that $x_2$ must lie above the optimal value $x^\star = \frac{1}{\ln a}\arcsinh\frac{1}{2\ln a}$. The upper bound is due to the fact that the function $f(x)$ grows faster for $x>x^\star$ than for $x<x^\star$, i.e., $x^\star-2>x_2-x^\star$. Thus, $x_2<2x^\star -2$. 

In case we know $x_1$, the more accurate bounds on $x_2$ follows similarly. Since\break $2(x_2-x^\star)>x^\star -x_1$, the lower bound is clearly $x_2>\frac{3}{2}x^\star-\frac{1}{2}x_1$. The upper bound comes from fact that $x_2-x^\star<x^\star-x_1$.
\end{proof}

Another insight on the solution space can be viewed on Fig.~\ref{fig:solspace}. Observe the minimum at $a=1, x=2$ (not considering case $a=0, x=0$) and the slight asymmetricity of the curve around the line $a=1$.

\begin{figure}\centering
\psfrag{a}[Bl][Bl][1][0]{$0.7$}
\psfrag{b}[Bc][Bc][1][0]{$0.8$}
\psfrag{c}[Bc][Bc][1][0]{$0.9$}
\psfrag{d}[Bc][Bc][1][0]{$1$}
\psfrag{e}[Bc][Bc][1][0]{$1.1$}
\psfrag{f}[Bc][Bc][1][0]{$1.2$}
\psfrag{g}[Bc][Bc][1][0]{$1.3$}
\psfrag{h}[Bc][Bc][1][0]{$1.4$}
\psfrag{A}[Br][Br][1][0]{$0$}
\psfrag{B}[Br][Br][1][0]{$5$}
\psfrag{C}[Br][Br][1][0]{$10$}
\psfrag{D}[Br][Br][1][0]{$15$}
\psfrag{E}[Br][Br][1][0]{$20$}
\psfrag{F}[Br][Br][1][0]{$25$}
\psfrag{G}[Br][Br][1][0]{$30$}
\psfrag{H}[Br][Br][1][0]{$35$}
\psfrag{I}[Br][Br][1][0]{$40$}
\psfrag{x}[Bc][Bc][1][0]{$a$}
\psfrag{y}[Bc][Bc][1][0]{\raisebox{0.3cm}{$x$}}
\psfrag{R}[Bc][Bc][1][0]{I.}
\psfrag{Q}[Bc][Bc][1][0]{II.}
\includegraphics[width=9.5cm]{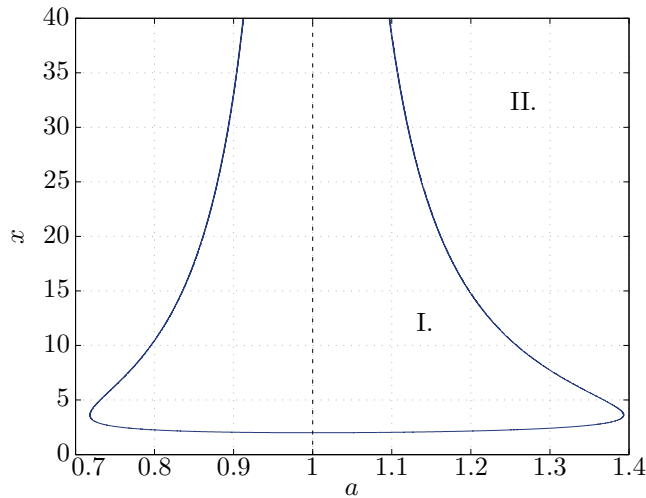}
\captionsetup{skip=0.1cm}
\caption{Solution space of Eq.~\eqref{eq:mainf}.\! Region~I. depicts $f(x)\!<\!0$, Region~II. depicts $f(x)\!>\!0$. Additionally, there exists an isolated point at $a=0$, $x=0$ (not shown).}\label{fig:solspace}
\vspace{-0.1cm}
\end{figure}

\subsection{Simulation results}

As shown in Tab.~\ref{tab}, in general, the bounds are reasonable as long as $a\in (e^{-\frac{1}{2\sinh q}}, e^{\frac{1}{2\sinh q}})$, i.e., $a\in(0.71793825, 1.39287744)$. Also it is obvious that if we know $x_1$ then the bounds on $x_2$ are tighter (``2nd bounds'') than if we do not know it (``1st bounds'').

\begin{table}[h!]\centering
\begin{tabular}{lll|p{2cm}p{2cm}|p{2cm}p{2cm}l}
\hline
$a$ & $x_1$ & $x_2$ & \centering 1st lower bound on $x_2$ & \centering 1st upper bound on $x_2$ & \centering 2nd lower bound on $x_2$ & \centering 2nd upper bound on $x_2$&\hspace*{-0.5cm}\\
\hline
$0.6$ & -- & -- & $1.6959$ & $1.3918$ & -- & --&\hspace*{-0.5cm}\\
$0.75$ & $2.5738$ & $6.3160$ & $4.5882$ & $7.1764$ & $5.5954$ & $6.6026$&\hspace*{-0.5cm}\\
$0.9$ & $2.0467$ & $33.2488$ & $21.4624$ & $40.9248$ & $31.1702$ & $40.8781$&\hspace*{-0.5cm}\\
$1.08$ & $2.0243$&  $51.1120$ & $33.3978$ & $64.7955$ & $49.0845$ & $64.7712$&\hspace*{-0.5cm}\\
$1.39$ & $3.3144$ & $3.9932$ & $3.6589$ & $5.3179$ & $3.8312$ & $4.0035$&\hspace*{-0.5cm}\\
\hline
\end{tabular}
\caption{Simulations results for various $a$ and bounds on $x_2$. It is obvious that if there are two roots then the bounds from Lemma~\ref{lem2} hold. Notice that as $a\to 1$ the bounds get looser. As $a\to e^{1/2\sinh(q)}\approx 1.3928$ or $a\to e^{-1/2\sinh(q)}\approx 0.7179$ the bounds are tighter.}\label{tab}
\end{table}

\vspace{-0.2cm}
\section{Conclusions}

We analyzed the solutions of equation $a^x+a^{-x}=x$, which, to our knowledge, has not been mentioned in the literature. We gave strict conditions on $a$ as well as on the roots. We showed that there are always at most two solutions. For these solutions we gave bounds which depend on $a$. These bounds may be very useful for numerical enumeration of the roots, when used as intialization values.

Nevertheless, the question if there is an analytical solution on $x_1$ and $x_2$ remains unsolved.

\subsection*{Acknowledgements}
The author would like to thank Claude F. Leibovici for all the valuable inputs.

\bibliographystyle{unsrt}
\bibliography{references}

\end{document}